\documentclass[12pt]{report}
\usepackage{makeidx}
\makeindex

\usepackage{amssymb,amsmath,amsthm}
\usepackage{epsfig,amssymb,amsmath,amsthm,graphicx,psfrag}
\usepackage[all]{xy}
\usepackage{enumerate}

\newtheorem{lemma}{Lema}[chapter]

\newtheorem{theorem}[lemma]{Teorema}
\newtheorem{remark}{Observaci\'on}
\newtheorem{notation}{Notaci\'on}

\newtheorem{corollary}[lemma]{Corolario}

\newtheorem{definition}[lemma]{Definici\'on}

\begin{document}
\section{Orbits of $\mathbb{Z}_N ^k$ under the action of $S_k$ } \label{C:oua}

In this Section we follow Saldarriaga \cite{S}, see also \cite{FW}. Let $G$ be the group $\mathbb{Z}_N ^k$, and for each $N$-tuple of
nonnegative integers
\[ (a_0,a_1,...,a_{N-1}) \qquad \text{ such that} \qquad a_0+a_1+...+a_{N-1}=k \]
we define the subset of $G$
\begin{equation} \label{E:Part}
[(a_0,a_1,...,a_{N-1})]=\left\{ x\in\mathbb{Z}_N ^k \mid j \text{ occurs }  a_j \text{ times in } x \text{, }
0\leq j\leq N-1 \right\}.
\end{equation}
Then $G$ is a disjoint union of these subsets.

Note that the symmetric group $S_k$ acts on $G$ by permuting $k$-tuples, the \textbf{set of orbits under this action},
$\mathcal{O}=\mathcal{O}(N,k)$,  consists of the subsets \eqref{E:Part} defined above, and each orbit contains a unique
representative in \textbf{standard form}
\begin{equation} \label{E:orb}
 \left( (N-1)^{a_{N-1}},...,1^{a_1},0^{a_0}\right)
\end{equation}
where the exponent indicates the number of repetitions of the base.

\begin{notation}
Given $x\in \mathbb{Z}_N ^k$, we will denote the \textbf{orbit of x} by $[x]$ and the \textbf{representative in standard
form of this orbit} will be denoted by ${\hat x}$.
\end{notation}

For orbits $[a]$,$[b]$ and $[c]$, we define the set:
\[ T([a],[b],[c])=\{(x,y,z)\in [a]\times [b]\times [c] \mid x+y=z \}. \]
Note that $\sigma\in S_k$ acts on $(x,y,z)\in T([a],[b],[c])$ by $\sigma(x,y,z)=(\sigma x,\sigma y,\sigma z)$.
\begin{definition}
Denote by $\mathbf{M_{[a],[b]}^{(k)[c]}}$\textbf{ the number of }$\mathbf{S_k}$-\textbf{orbits of }$\mathbf{T([a],[b],[c])}$.
\end{definition}

This suggests that we could use these numbers as the structure constants of an algebra
(depending on $N$ and $k$ and over any field of characteristic 0) with basis $\mathcal{O}$, by defining a bilinear product
as follows:
\begin{equation}  \label{E:orbprod}
[a]\times [b]=\sum_{c\in \mathcal{O}}M_{[a],[b]}^{(k)[c]} [c].
\end{equation}

The following is a description of how to compute the product of two $S_k$-orbits of $\mathbb{Z}_N ^k$.

\begin{definition}
Let $[a]$, $[b]\in\mathcal{O}$, and assume that $[b]=\{y_1,...,y_t \}$. For $1\leq i\leq t$ set:
\begin{equation} \label{E:list}
 z_i=\hat a +y_i .
\end{equation}
We say that the equation $z_j=\hat a +y_j$ in the list \eqref{E:list} is \textbf{redundant}, if for some $i<j$ and some
$\sigma\in S_k$ we have
\[ \sigma\hat{a} =\hat{a} \text{, }\sigma y_j=y_i\text{ and } \sigma z_j=z_i ,\]
that is, if the triples $(\hat{a} ,y_i,z_i)$ and $(\hat{a} ,y_j,z_j)$ are in the same $S_k$-orbit of $T([a],[b],[z_i])$.
\end{definition}

Then the product of two orbits can be computed as follows: let $[a]$, $[b]\in\mathcal{O}$ and fix a representative
from the orbit $[a]$, say the representative in standard form, $\hat a$, and assume that $[b]=\{y_1,...,y_t \}$. For every
$y_i\in [b]$, set
\[ z_i=\hat a +y_i , \qquad 1\leq i\leq t . \]
Remove all redundant equations from this list, and without loss of generality, assume that after removing all redundancies,
we are left with the first $s$ equations, for some $s\leq t$. That is, the list:
\[ z_i=\hat a +y_i , \qquad 1\leq i\leq s,  \]
has no redundancies. Then
\[ [a]\times [b]=[z_1]+[z_2]+...+[z_s] . \]
Note that several $z_i$'s could be in the same orbit, and for every $[c]\in\mathcal{O}$,
\begin{equation} \label{E:M1}
 M_{[a],[b]}^{(k)[c]}=Card\{1\leq i\leq s \mid z_i\in [c] \} .
\end{equation}

\begin{remark} \label{R:M_[a],[b]^[c]}
From Equation \eqref{E:M1}, we also get that $M_{[a],[b]}^{(k)[c]}$ can be computed by removing all redundancies from the
list of equations
\[  z=\hat a +y, \]
where $y\in[b]$ and $z\in[c]$.
\end{remark}

The following Theorem was proved in \cite{S}

\begin{theorem} \label{T:orbitproduct}
Let $[a],[c]\in\mathcal{O}$ and $[b]=[(1^m,0^{k-m})]$ for some $m\leq k$. Suppose that $M_{[a],[b]}^{(k)[c]}\ne 0$ then
$M_{[a],[b]}^{(k)[c]}=1$.
\end{theorem}

\begin{notation} If $a=(a_1,\dots,a_k)\in\mathbb{Z}_N^k$, we will denote for $(a,0)$ the $k+1-$tuple $(a_1,\dots,a_k,0)\in\mathbb{Z}_N^{k+1}$ and for $[a,0]$ its corresponding orbit.
\end{notation}

\begin{theorem} \label{T:levelincreasingfororbits} Let $a,b,c\in\mathbb{Z}_N^k$ with $b$ so that $\hat{b}=(1^m,0^{k-m})$ then $M_{[a],[b]}^{(k)[c]}\leq M_{[a,0],[b,0]}^{(k+1)[c,0]}$
\end{theorem}
\begin{proof} It is clear that $M_{[a],[b]}^{(k)[c]}$ is a non-negative number, hence if $M_{[a],[b]}^{(k)[c]}=0$ the result follows.

Now, assume that $M_{[a],[b]}^{(k)[c]}\ne0$, then from Theorem \ref{T:orbitproduct} we get that $M_{[a],[b]}^{(k)[c]}=1$. Hence from Remark \ref{R:M_[a],[b]^[c]} we get that there exists $y\in[b]$ and $z\in[c]$ so that $\hat{a}+y=z,$
therefore we get that
\begin{equation}\label{E:extension}
(\hat{a},0)+(y,0)=(z,0).
\end{equation}

It is clear that $\hat{(a,0)}=(\hat(a),0)$, $(\hat(a),0)\in[a,0]$, $(y,0)\in[y,0]$ and $(z,0)\in[z,0]$. Hence, from Equation \eqref{E:extension} and Remark \ref{R:M_[a],[b]^[c]} we get that $M_{[a,0],[b,0]}^{(k+1)[c,0]}\ne0$ and then the result follows.
\end{proof}

\section{Particular case of the level increasing conjecture}

Let $\lambda=a_1\lambda_1+\cdots+a_{N-1}\lambda_{N-1}$ a weight for $A_{N-1}$ of level $k$, hence $a_1+\cdots+a_{N-1}\leq k$, then obviously $a_1+\cdots+a_{N-1}\leq k+1$ and $\lambda$ can also be considered as a weight of level $k+1$. Notice that the corresponding orbit $[\lambda]$ in $\mathbb{Z}_N^{k}$ associated to the weight $\lambda$ is given by

$[\lambda]=\left[\left((N-1)^{a_{N-1}},(N-2)^{a_{N-2}},\dots,1^{a_1},0^{k-\sum_{i=1}^{N-1}a_i}\right)\right]$ and if we see $\lambda$ as a weight of level $k+1$ then its corresponding orbit is
\[ \left[\left((N-1)^{a_{N-1}},(N-2)^{a_{N-2}},\dots,1^{a_1},0^{k+1-\sum_{i=1}^{N-1}a_i}\right)\right]=[\lambda,0]\]
In other words, if $[a]$ is the corresponding orbit to $\lambda$ in $\mathbb{Z}_N^k$, then $[a,0]$ is the corresponding orbit in $\mathbb{Z}_N^{k+1}$ of $\lambda$ seen as a weight of level $k+1$.

The following theorem was proved in \cite{S}

\begin{theorem} \label{T:main}
Let $\lambda=m\lambda_1$ be a multiple of the first fundamental weight for $A_{N-1}$, $m\leq k$,
$\mu=a_1\lambda_1+...+a_{N-1}\lambda_{N-1}$ any other weight of level $k$ and $[\lambda]$ and $[\mu]$ their
corresponding orbits in $\mathbb{Z}_N ^k$. Then $N_{\mu,\lambda}^{(k)\nu}=M_{[\mu],[\lambda]}^{(k)[\nu]}$ for any weight
$\nu$ of level $k$.
\end{theorem}

So we get the following special case of the Level Increasing Conjecture

\begin{corollary} \label{T:levelincreasingconjectureforrows} \textbf{(Special case of the Level Increasing Conjecture)} Let $\lambda, \mu$ and $\nu$ weights for $A_{N-1}$ of level $k$ with $\lambda=m\lambda_1$ a multiple of the first fundamental weight, then we get
\[ N_{\mu,\lambda}^{(k)\nu}\leq N_{\mu,\lambda}^{(k+1)\nu}.\]
\end{corollary}
\begin{proof} From Theorem \ref{T:main} we get $N_{\mu,\lambda}^{(k)\nu}=M_{[\mu],[\lambda]}^{(k)[\nu]}$ and from Theorem \ref{T:levelincreasingfororbits} we get that $M_{[\mu],[\lambda]}^{(k)[\nu]}\leq M_{[\mu,0],[\lambda,0]}^{(k+1)[\nu,0]}$ and since
$M_{[\mu,0],[\lambda,0]}^{(k+1)[\nu,0]}=N_{\mu,\lambda}^{(k+1)\nu}$ we get
\[ N_{\mu,\lambda}^{(k)\nu}\leq N_{\mu,\lambda}^{(k+1)\nu}.\]
\end{proof}

\end{document}